\documentclass[12pt]{amsart}

\usepackage{amssymb}
\usepackage{latexsym}
\usepackage{graphicx}
\usepackage{tikz}
\usetikzlibrary{arrows}

\theoremstyle{plain}
\newtheorem{prop}{Property}[section]
\newtheorem{theorem}[prop]{Theorem}
\newtheorem{proposition}[prop]{Proposition}
\newtheorem{corollary}[prop]{Corollary}
\newtheorem{lemma}[prop]{Lemma}
\newtheorem{conjecture}{Conjecture}
\newtheorem{claim}{Claim}

\theoremstyle{remark}

\def\x{\mathbf x}

\def\e{\mathbf e}

\def\vol{{\rm{vol}}}

\newcommand{\Z}{{\mathbb Z}}
\newcommand{\N}{{\mathbb N}}

\newcommand{\beeq}{\begin{eqnarray*}}
\newcommand{\eneq}{\end{eqnarray*}}

\newcommand{\be}{\begin{equation}}
\newcommand{\ee}{\end{equation}}

\definecolor{darkred}{cmyk}{.3,.9,.80,.2}

\title[Additive Volume of Sets Contained in Few Arithmetic Progressions]{Additive Volume of Sets Contained \\ in Few Arithmetic Progressions}

\author{Gregory A. Freiman}
\address{ 
	The Raymond and Beverly Sackler Faculty of Exact Sciences\\
	School of Mathematical Sciences\\
	Tel Aviv University}
\email{grisha@post.tau.ac.il}

\author{Oriol Serra}
\address{Department of Mathematics, Universitat Polit\`ecnica de Catalunya and\\ 
	Barcelona Graduate School of Mathematics, Barcelona}
\email{oriol.serra@upc.edu}
\thanks{O.\,Serra was supported by the Spanish Ministerio de Econom\'ia y Competitividad under project MTM2014-54745-P}

\author{Christoph Spiegel}
\address{Department of Mathematics, Universitat Polit\`ecnica de Catalunya and\\ 
	Barcelona Graduate School of Mathematics, Barcelona}
\email{christoph.spiegel@upc.edu}
\thanks{C.\,Spiegel was supported by the Spanish Ministerio de Econom{\'i}a y Competitividad FPI grant under the project MTM2014-54745-P, the project MTM2017-82166-P and the Mar{\'i}­a de Maetzu research grant MDM-2014-0445.}
\date{\today}

\begin{document}

\maketitle

\begin{abstract}
	A conjecture of Freiman gives an exact formula for the largest volume of a set of integers $A \subset \Z$ with given cardinality $k = |A|$ and doubling $T = |2A|$. The formula is known to hold when $T \le 3k-4$, for some small range over $3k-4$ and for families of structured sets called chains. In this paper we extend the formula to sets of every dimension and prove it for sets composed of three segments, giving structural results for the extremal case. A weaker extension to sets composed of a bounded number of segments is also discussed.
\end{abstract}

\section{Introduction}

Let $A\subset \Z$ be a finite set of integers. The {\it Minkowski sum} of $A$ is $A+A=\{a+a': a,a'\in A\}$. The {\it doubling} of $A$ is the cardinality of $2A=A+A$. The Freiman--Ruzsa theorem giving the structure of sets of integers with small doubling is one of the central results in  Additive Number Theory. It states that a set $A$ with doubling $|2A|\le c|A|$ is a dense set of a multidimensional arithmetic progression $P$, where the density $|A|/|P|$ and dimension of $P$ depend only on $c$, see Freiman~\cite{Freiman87},  Bilu~\cite{Bilu99} and Ruzsa~\cite{Ruzsa94}. The estimation of the best lower bounds for the density of $A$ in $P$ was the object of a long series of papers and it was eventually brought to its essentially best values by Schoen~\cite{Schoen2011}.

At a conference in Toronto in 2008, Freiman proposed a precise formula for the largest possible volume of a set of integers $A \subset \Z$ with given doubling $T=|2A|$ in terms of a specific parametrization of the value of $T$, see~\cite{Freiman2014}. We start by recalling some definitions in order to state this conjecture.

Given abelian groups $G$ and $G'$, two sets $A\subset G$ and $B\subset G'$ are {\it Freiman isomorphic of order $2$} ($F_2$--isomorphic for short) if there is a bijection $\phi:A\to B$ such that, for every $x,y,z,t \in A$, we have 
\begin{equation}
	 x+y=z+t \quad \Leftrightarrow \quad \phi(x)+\phi(y)=\phi(z)+\phi(t).
\end{equation}

The {\it additive dimension} $\dim (A)$ of a set $A \subset \Z$ is the largest $d \in \N$ such that there exists a set $B \subset \Z^d$ not contained in a hyperplane of $\Z^d$ which is $F_2$--isomorphic to $A$. Note that any $d$--dimensional set $A$ of cardinality $k$  satisfies $d \leq k-1$ and that furthermore by results of Freiman~\cite{Freiman73} as well as Konyagin and Lev~\cite{KL2000} we have
\begin{equation} \label{eq:T_bounds}
	(d+1)k-{d+1 \choose 2} \leq |2A| \leq {k \choose 2}+d+1.
\end{equation}

The {\it volume} $\vol (A)$ of a $d$--dimensional set $A$ is defined to be the minimum cardinality of the convex hull among all sets in  $\Z^d$ that are $F_2$--isomorphic to $A$.

We say that a set of integers $A \subset \Z$ is in {\it normal form} if $\min (A)=0$ and $\gcd (A)=1$. We call
	\begin{equation}
		\tilde{A} = \big( A-\min (A) \big) / \gcd \big( A-\min (A) \big)
	\end{equation}
	the {\it normalization} of $A$ since it is a set in normal form and it is $F_2$--isomorphic to $A$. Note that for any $1$--dimensional set $A$ we have
	\begin{equation}
		\vol(A) = \max(\tilde{A})+1.
	\end{equation}
	If $\min(A)=0$, then the {\it reflection} of $A$ is defined as $A^-=-A+\max(A)$. The reflection of $A$ is certainly isomorphic to $A$.

We are interested in obtaining upper bounds for the volume of a set $A$ of integers in terms of its cardinality $|A|$, the cardinality of its doubling $|2A|$ and its dimension $\dim(A)$. We denote the maximum volume of all sets $A$ of integers with cardinality $k$, doubling $T$ and dimension $d$ by
\begin{equation}
	\vol(k,T,d) = \max \{\vol (A): A\subset \N \;, |A|=k, |2A|=T \text{ and } \dim(A)=d\}.
\end{equation}
A set $A$ is {\it extremal} if  $\vol(A) = \vol(|A|,|2A|,\dim(A))$. The following is a more general and slightly reformulated version of the previously mentioned conjecture of Freiman, which can be traced back to \cite{Freiman73}. Its notable addition is that it takes the dimension of a set into consideration.

\begin{conjecture}\label{conj:main}
	Given any $k,T,d \in \N$ such that
	\begin{equation}\label{eq:2}
		T = (d+c) k - {d+c+1 \choose 2} + d + b+1,
	\end{equation}
	where 
	\begin{equation}
		1 \leq c \leq k-d-1 \, \text{ and } \, 0 \leq b \leq k-d-c-1,
	\end{equation}
	we have
	\begin{equation}\label{eq:max}
		\vol (k,T,d) = 2^{c-1} \, (k-c+b)+1.
	\end{equation}
\end{conjecture}

Note that given any $k \in \N$, $1 \leq d \leq k-1$ and $T \in \left\{ (d+1)k-{d+1 \choose 2},\dots,{k \choose 2}+d+1 \right\}$, there are uniquely defined 
	\begin{equation}
		c = c(k,T,d) \quad \text{and} \quad b = b(k,T,d)
	\end{equation}
subject to the boundary conditions stated in the conjecture, such that $T$ can be expressed as the right--hand side of \eqref{eq:2}. It follows that Conjecture~\ref{conj:main} states a tight upper bound on the volume of any possible set of integers. If $A$ has cardinality $|A| = k$, dimension $\dim(A) = d$ and doubling $|2A|=T$ then we call this uniquely determined $c=c(k,T,d)$ the {\it doubling constant} of $A$.

There are  examples showing that the right hand side in \eqref{eq:max} is at least a lower bound for $\vol(k,T,d)$, see e.g.~\cite{Freiman73, FS2017}. Thus an extremal set $A$ has volume
$$
vol(A)\ge vol(k,T,d).
$$
Furthermore, equality has been established in a few cases:
\begin{enumerate}
	\item	by Freiman~\cite{Freiman73} for one--dimensional sets satisfying $T \leq 3k-4$, that is either $c = 1$ and any admissible $b$ or $c = 2$ and $b = 0$, with an additional structural description of extremal sets given in~\cite{Freiman2009},
	\item	by Freiman~\cite{Freiman73} as well as by Hamidoune and Plagne~\cite{HamPlagne2001} for one--dimensional sets if $T = 3k-3$, that is $c = 2$ and $b = 1$, with a structural description of the extremal case due to Jin~\cite{Jin2007},
	\item	by Freiman~\cite{Freiman73} for two--dimensional sets satisfying $k \geq 10$ and $T \leq 10/3 \, k - 6$, that is $c = 1$ and $0 \leq b \leq k/3-2$, with a structural description of any such set,
	\item	by Jin~\cite{Jin2008} using tools from non-standard analysis in the case of large one--dimensional sets satisfying $T \leq (3+\epsilon) k$, that is $c = 2$ and $0 \leq b \leq \epsilon k$, for some $\epsilon > 0$,
	\item	by Stanchescu~\cite{Stanchescu2010} for any $d$--dimensional set satisfying $c = 1$ and $b = 0$,
	\item	by Freiman and Serra~\cite{FS2017} for a class of one--dimensional sets called chains, which can be seen as extremal sets build by a greedy algorithm, and any admissible values of the doubling constant $c$.
\end{enumerate}
  
In order to give further evidence towards the validity of this conjecture, we consider sets composed of a given number of segments. Throughout the paper we say that   $A \subset \Z$  is  the union of $s$  segments if 
\begin{equation}\label{eq:ssegments}
	A = P_1 \cup \cdots \cup P_s
\end{equation}
where each $P_i$ is a segment of length $k_i$ with  $\max P_i +1< \min P_{i+1}$ for $1 \leq i < s$ and moreover  $k_i > 1$ for some $1 \leq i \leq s$. Regarding the doubling of such a  set, we have the upper bound 
\begin{equation}\label{eq:doublingsegments}
	|2A|=\left| \cup_{1\le i\le j\le s} (P_i+P_j)\right|\le \sum_{1\le i\le j\le s} (|P_i|+|P_j|-1)=(s+1)|A|-{s+1\choose 2}.
\end{equation}
Equality holds if and only if the sums $P_i+P_j$ are pairwise distinct, in which case the set $A$ has dimension $s$.

As previously mentioned, Conjecture~\ref{conj:main} has been proved for sets $A$ with doubling $|2A|\le 3|A|-3$, and the structure of extremal sets for this range of doubling is well understood. In spite of many efforts, not much is known about the exact  maximum volume of  sets with doubling at least $3|A|-2$. This motivates us to consider sets composed of three segments and doubling larger than $3|A|-3$.  Our main result is to show that the statement of Conjecture~\ref{conj:main} holds for sets $A$ composed of three segments, also giving a structural description of the extremal cases.

\begin{theorem}\label{thm:3seg}
	Let $A$ be an extremal  set with cardinality $k>7$ and doubling $|2A|>3k-4$  consisting of three segments. 
	Then $A$ is isomorphic to one of the following sets:
	\begin{enumerate}
		\item[(i)] $([0,k+b-2]\setminus [1,b])\cup \{2(k+b-2)\}$ with  $\dim (A)=1$,  $\vol (A)=2k+2b-3$ and $|2A|=3k-4+b$, 
		\item[(ii)] $([0, k+b+i-3]\setminus [k-i-1,k+b-3]) \cup \{2(k+b-2)\}$ for $1\le i\le k/3$ and $2 \leq b \leq k-2i-1$, with  $\dim (A)=1$, $\vol (A)=2k+2b-3$ and  $|2A|=3k-4+b$, 
		\item[(iii)] $([0,k+b-2]\setminus [1,b])\times\{0\}  \cup \{(0,1)\}\subset \Z^2$  with  $\dim (A)=2$,  $\vol (A)=k+b$ and  $|2A|=3k-3+b$, or
		\item[(iv)] $([0,k_1-1]\times \{(0,0)\})\cup ([0,k_2-1]\times \{(0,1)\})\cup ([0,k_3-1]\times \{(1,0)\})\subset \Z^3$  with $k_1+k_2+k_3=k$, $k_1, k_2, k_3\ge 1$, $\dim (A)=3$, $\vol (A)=k$ and  $|2A|=4k-6$.
	\end{enumerate}
\end{theorem}

The  extremal sets composed of three segments described in Theorem~\ref{thm:3seg}  are illustrated for $k=11$ and $|2A|=3k-1$ in Figure~\ref{fig:extr11}.

\begin{figure}[h]
\begin{center}	
 	\begin{tikzpicture}[scale=0.6]
	\foreach \x in {0,..., 24}
  	{
	\foreach \j in {1,...,4}
  		\draw (\x, \j) circle(3pt);
  	}

	\foreach \i in {0,...,24}
	\node[above] at (\i,4) {\tiny $\i$};
  	
	\foreach \x in {0, 4, 5, 6, 7, 8, 9,10,11,12, 24}
  	{
  		\draw[fill] (\x, 4) circle(3pt);
  	}

  	\foreach \x in {0, 1, 2, 3, 4, 5, 6, 7, 8, 12, 24}
  	{
  		\draw[fill] (\x, 3) circle(3pt);
  	}
  	\foreach \x in {0, 1, 2, 3, 4, 5, 6, 7, 12, 13, 24}  	
	{
  		\draw[fill] (\x, 2) circle(3pt);
  	}

	\foreach \x in {0, 1, 2, 3, 4, 5, 6, 12, 13, 14, 24}
	{
  		\draw[fill] (\x, 1) circle(3pt);
  	}

  	\end{tikzpicture}
 \end{center}
 	
	\
	
	\begin{center}
	\begin{tikzpicture}[scale=0.6]
  	\draw [help lines] (0,1) grid (11,2);
	
	\foreach \x in {0, 3,4, 5, 6, 7, 8, 9,10,11}
  	{
  		\draw[fill] (\x, 1) circle(3pt);
  	}
	\draw[fill] (0,2) circle(3pt);
	
\end{tikzpicture}
\end{center}

\caption{Extremal sets, up to isomorphism,  for $k=11$ and $|2A|=3k-1$.}\label{fig:extr11}
\end{figure}
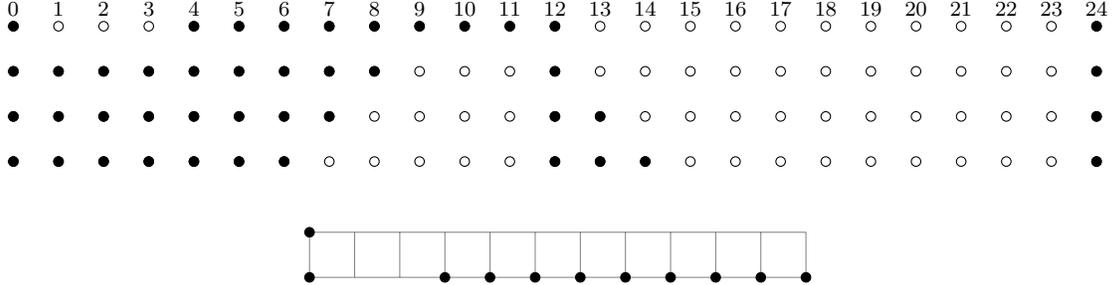

\medskip

We also believe that the following statement, regarding an upper bound on the volume of $1$--dimensional sets composed of few disjoint segments, should hold as well. It is independent of the doubling of such sets and in fact one may derive it from Conjecture~\ref{conj:main} as the case with maximum doubling without too much effort.

\begin{conjecture} \label{conj:segmentFreiman}
	Let $A$ be a $1$--dimensional set and let $s$ be the minimum number of disjoint segments into which it can be decomposed. If $s \leq |A|-1$, then
	\begin{equation} \label{eq:segment_covering}
		\vol(A)  \leq 2^{s-1} \big( |A| - s \big)+1.
	\end{equation}
\end{conjecture}

The set $A_s=\{0,1,\ldots ,k-s\}\cup \{2^i(k-s), i=1,\ldots ,s-1\}$ shows that this conjectured upper bound would be tight. Note that, if we only know that a set $A$ is $1$--dimensional, Freiman~\cite{Freiman73} gave $\vol (A) \leq 2^{k-2}+1$ as an upper bound on the volume of $A$. In this paper, we show the validity of Conjecture~\ref{conj:segmentFreiman} for some small values of $s$.
\begin{proposition}\label{prop:small_segments}
	The statement of Conjecture~\ref{conj:segmentFreiman} holds for $s \leq 4$.
\end{proposition}
One may prove Conjecture~\ref{conj:segmentFreiman} for further moderate values of $s$ along the lines presented in the proof of Propositon~\ref{prop:small_segments}, at the cost of a more involved case analysis. A general proof of this conjecture avoiding the increase of cases would be of interest.

\medskip

\noindent {\bf Outline. } The paper is organized as follows: in Section~\ref{sec:dimension} we discuss the dimension of sets composed of few segments. We then we prove Proposition~\ref{prop:small_segments} in Section~\ref{sec:prop_proof} and Theorem~\ref{thm:3seg} in Section~\ref{sec:thm_proof}. 

\section{The dimension of sets contained in few segments} \label{sec:dimension}  

Konyagin and Lev~\cite{KL2000} established a formula for the dimension of a given set $A \subset \Z^m$ of cardinality $k$. Let us write $A = \{a_1, \dots, a_k\}$ and introduce some necessary notation. For $1 \leq i \leq k$, let $\e_i$ denote the vector in $\mathbb{R}^k$ that has a one at coordinate $i$ and zero everywhere else. $M_A$ denotes the integer valued matrix with $k$ columns obtained by listing as its rows  all vectors $\e_{i_1} + \e_{i_2} - \e_{i_3} - \e_{i_4}$ for which $a_{i_1} + a_{i_2} = a_{i_3} + a_{i_4}$ holds and for which we do not have $i_1 = i_2 = i_3 = i_4$. The dimension of $A$ can be derived from the rank of $M_A$ using the following result.
\begin{theorem}[Konyagin and Lev~\cite{KL2000}]\label{thm:kl}
	For any set $A \subseteq \mathbb{Z}^m$ we have
	\begin{equation}
		\dim(A) = |A| - 1 - \rm{rank}(M_A).
	\end{equation}
\end{theorem}

Now let $A \subset \Z$ be a set which is the union of $s$ disjoint  segments as in \eqref{eq:ssegments}. 
Given such a set $A$, we denote by $S_A$ the integer valued matrix  with $s$ columns  obtained by listing in its rows all vectors $\e_{j_1} + \e_{j_2} - \e_{j_3} - \e_{j_4}$ for which $(A_{j_1} + A_{j_2}) \cap (A_{j_3} + A_{j_4} )\neq \emptyset$ and for which we do not have $j_1 = j_2 = j_3 = j_4$. We derive the following Corollary from Theorem~\ref{thm:kl}.

\begin{corollary} \label{cor:segdim}
	Any $A \subset \Z$ that is the union of $1 \leq s \leq |A| - 1$ disjoint segments satisfies
	\begin{equation}
		\dim(A) = s - \rm{rank}(S_A).
	\end{equation}
\end{corollary}

\begin{proof}
	Every row in $M_A$ is associated with up to four (not all equal) elements $a_{i_1},a_{i_2},a_{i_3},a_{i_4}$ such that $a_{i_1} + a_{i_2} = a_{i_3} + a_{i_4}$. Let $a_{i_1} \in P_{j_1}$, $a_{i_2} \in P_{j_2}$, $a_{i_3} \in P_{j_3}$ and $a_{i_4} \in P_{j_4}$ where  we may assume  $j_1 \leq j_2$ and $j_3 \leq j_4$ as well as $\min\{j_1,j_2\} \leq \min\{j_3,j_4\}$. Furthermore, let $0 \leq y = \# \{ j : |P_j| = 1 \} < s$ denote the number of segments that are singletons. We distinguish the following cases.
	
	{\bf Case 1.} $\# \{j_1,j_2,j_3,j_4\} = 1$, that is $j_1 = j_2 = j_3 = j_4 = j$ for some $1 \leq j \leq s$. Since $P_j$ is one dimensional, by Theorem~\ref{thm:kl}  there are a total of $\max \{|P_j| - 2,0\}$ linearly independent equations of this type for each $P_j$.  As these equations only involve elements in $P_j$ and the segments are disjoint, it is clear that each equation is linearly independent from those of other segments, so we get a total of $|A| - 2s + y$ linearly independent equations of this type in $M_A$.  On the other hand this case does not contribute to the rank  of $S_A$
	
	{\bf Case 2.} $\# \{j_1,j_2,j_3,j_4\} = 2$. We distinguish two further cases.
	
	{\bf Case 2.1.}  $j_1 = j_3 < j_4 = j_2$.  Segments of length one can only give trivial equations of this type not contributing to the rank of $M_A$. If $P_{i_0}<P_{i_1}<\cdots <P_{i_{s-y}}$ are  the  $s-y$ segments which are not singletons,  then equations of this type give us a total of $s - y - 1$ new linear independent ones on top of the ones given by Case 1, one for each pair $P_{i_0}, P_{i_j}$ and $j=2,\ldots s-y$, the remaining ones being linearly dependent with these. Moreover this case does not contribute to the rank of $S_A$. 
	
	{\bf Case 2.2.} $j_1<j_2=j_3=j_4$ or $j_1=j_2=j_3<j_4$. Each pair $P_{j}, P_{j'}$ with $j\neq j'$ for which an equation of this type  exists implies that $P_j\cap P_{j'}$ intersects either $2P_j$ or $2P_{j'}$ and   contributes one additional linear independent equation in $M_A$ on top of the above ones, and it contributes to one additional linear equation in $S_A$ as well. 
		
	{\bf Case 3.} $\# \{j_1,j_2,j_3,j_4\} \geq 3$, that is $j_1 < j_3 \leq j_4 < j_2$. This implies that $P_{j_1} + P_{j_2}$ intersects $P_{j_3} + P_{j_4}$. Each such intersection contributes with one additional linear independent equation in $M_A$ and also on $S_A$ on top of the above ones. 
		
	Taken together it follows that ${\rm{rank}}(M_A) = (|A| - 2s + y) + (s - y - 1) + {\rm{rank}}(S_A)$ and therefore, by Theorem~\ref{thm:kl},  $\dim(A) = s - {\rm{rank}}(S_A)$.
\end{proof}

We note that for $s \geq 6$ there are sets for which $k_i = 1$ for all $i$ that are not covered by Corollary~\ref{cor:segdim} or Conjecture~\ref{conj:segmentFreiman}.

\section{Proof of Proposition~\ref{prop:small_segments}} \label{sec:prop_proof}

Let $A \subset \Z$ again be a set  which is the union of $s$ segments. We denote the interval separating the two consecutive segments $P_i$ and $P_{i+1}$ by
	\begin{equation}
		L_i = [\max(P_i)+1,\min(P_{i-1})-1]
	\end{equation}
	and write $\ell_i = |L_i|$ for its cardinality. It follows that
	\begin{equation}
		\min P_i = \sum_{j<i} (k_j+l_j) \quad \text{and} \quad \max (P_i) = \min (P_i)+k_i-1.
	\end{equation}
	
In order to prove Proposition~\ref{prop:small_segments}, we will need the following lemma that gives us an inductive approach to Conjecture~\ref{conj:segmentFreiman}.

\begin{lemma}\label{lem:pipi+1} Let $A$ be a $1$--dimensional set of cardinality $k$ that is composed of $s \leq k-1$ disjoint segments $P_1,\dots,P_s$ such that $\max P_i < \min P_{i+1}$ for $1 \leq i < s$. If Conjecture~\ref{conj:segmentFreiman} holds for $s-1$ and 
	\begin{equation}
		\vol (A)>2^{s-1}(|A|-s)+1,
	\end{equation}
	then we must have
	\begin{equation}\label{eq:aiaj+1}
		P_i+P_j<P_i+P_{j+1}, \qquad  1\le i\le s,\; 1\le j<s.
	\end{equation}
\end{lemma}

\begin{proof}
	We observe that, for each $1 \leq j < s$, the set $A \cup L_j$ is $1$--dimensional, consists of $s-1$ disjoint segments and has the same volume as $A$. By the assumption that Conjecture~\ref{conj:segmentFreiman} holds for $s-1$, we must have
	\begin{equation}
		\vol(A\cup L_j)\le 2^{s-2}(|A|+\ell_j-s+1)+1.
	\end{equation}
	Hence, our assumption on $\vol(A)$ implies that
	\begin{equation}\label{eq:ellj}
		\ell_j\ge |A|-s\ge k_i+1 \quad \text{ for }1\le i\le s,\; 1\le j<s.
	\end{equation}
	In particular, we have
	\begin{align*}
		\min (P_i+P_{j+1})-\max (P_i+P_{j})=&(\min P_{j+1}-\max P_j)-(\max P_i-\min P_i)\\
		=&\ell_j-k_i+1> 0,
	\end{align*}
	which implies the desired statement \eqref{eq:aiaj+1}.
\end{proof}

It follows that, under the hypothesis of Lemma~\ref{lem:pipi+1}, the only possible intersections between sums of two segments  are of the form
\begin{equation}\label{eq:iji'j'}
	(P_i+P_j)\cap (P_{i'}+P_{j'}) \quad \text{where} \quad i < i' \leq j' < j.
\end{equation}
By using \eqref{eq:iji'j'} we next prove Proposition~\ref{prop:small_segments}.

\begin{proof}[Proof of Proposition~\ref{prop:small_segments}]
	For $s=1$ the conjecture trivially holds.
	
	For $s = 2$, suppose that $A$ is composed of two segments and $\vol (A)>2k-2$. By Lemma~\ref{lem:pipi+1} we have $2P_1<P_1+P_2<2P_2$ and therefore ${\rm rank} (S_A) = 0$, so that $A$ must be $2$--dimensional by Corollary~\ref{cor:segdim}.
	
	Suppose now that $s=3$ and that $\vol(A)>4k-10$. By \eqref{eq:iji'j'} the only possible intesections of sums of segments are $P_1+P_3$ and $P_2$. Again, $A$ must be $2$--dimensional by Corollary~\ref{cor:segdim}.
	
	Finally, suppose that $s=4$ and $\vol (A) > 7k-30$.  By \eqref{eq:iji'j'} the only possible intersections of sets of the type $P_i + P_j$ are 
	\begin{enumerate}
		\item $P_1+P_3$ can intersect with $2P_2$,
		\item $P_1+P_4$ can only intersect with at most one of $2P_2$, $P_2+P_3$ or $2P_3$,
		\item $P_2+P_4$ can only intersect with $2P_3$.
	\end{enumerate}
	By Corollary~\ref{cor:segdim} the only case of interest is if three intersections occur, so let us distinguish the following two cases.
	
	{\bf Case 1.} $P_1+P_4$ intersects $P_2+P_3$. We note that the vector $(1,-1,-1,1)$ can be written as the sum of $(1,-2,1,0)$ and $(0,1,-2,1)$ and therefore Corollary~\ref{cor:segdim} again implies that $A$ is $2$--dimensional.
		
	{\bf Case 2.} $P_1+P_4$ intersects $2P_2$ or $2P_3$. It is clear that these cases are identical by symmetry, so let us assume the former. We must have
	\begin{align*}
	P_1+P_4 \cap 2P_2 \neq \emptyset   & \quad \Leftrightarrow \quad k_1 + k_2 + \ell_1 \geq k_3 + \ell_2 + \ell_3 + 2, \\
	P_1+P_3 \cap 2P_2 \neq \emptyset  & \quad \Leftrightarrow \quad k_2 + k_3 + \ell_2 \geq \ell_1 + 2, \\
	P_2+P_4 \cap 2P_3 \neq \emptyset & \quad \Leftrightarrow \quad k_3 + k_4 + \ell_3 \geq \ell_2 + 2.
	\end{align*}
	Combining the first two inequalities gives $\ell_3 \leq k_1 + 2k_2 - 4$ which combined with the third inequality gives $\ell_2 \leq |A| + k_2 - 6$ which inserted into the second inequality gives $\ell_1 \leq |A| + 2k_2 + k_3 - 8$. Taken together this would imply
	\begin{equation*}
	\ell = \ell_1 + \ell_2 + \ell_3 \leq 2|A| + k_1 + 5k_2 + k_3 - 18 \leq 7|A| - 31
	\end{equation*}
	in contradiction to the assumption that $\ell=\ell_1+\ell_2+\ell_3 \geq 7|A| - 30$.
\end{proof}

\section{Proof of Theorem~\ref{thm:3seg}} \label{sec:thm_proof}

We quote explicitly the so--called $(3k-4)$--Theorem of Freiman~\cite{Freiman73} mentioned in the Introduction which will be used throughout the proof.

\begin{theorem}[Freiman]\label{thm:3k-4} Let $A\subset \Z$ in normal form with $a=\max (A)$ and $k=|A|$. We have
$$
|2A|\ge \min \{ k+a, 3k-3\}.
$$
\end{theorem}

There are several versions of the above Theorem for the sum of distinct sets due to Freiman~\cite{Freiman73},  Lev and Smeliansky~\cite{LevSmel1995} and  Stanchescu~\cite{Stanchescu1996}.  We will use  the following slightly weaker form of the one  by Lev and Smeliansky~\cite{LevSmel1995}.

\begin{theorem}[Lev and Smelianski] \label{thm:ls}
	Any two finite sets $A, B \subset \Z$ in normal form with $\max (A) > \max (B)$ satisfy
	\begin{equation}
		|A+B|\ge \min\{ |A|+2|B|-2, \max (A)+|B|\}.
	\end{equation}
\end{theorem}

We make also use of the following result by Freiman~\cite{Freiman73}.

\begin{theorem}\label{thm:freiman10/3} Let $A$ be a two dimensional set of cardinality $k>6$ with $|2A| = 3|A| - 3 + b$. If $A$ can not be covered by a set consisting of two lines with volume at most $k+b$ then $b \geq |A|/3 - 2$.
\end{theorem}

We will also use the following Lemma which handles the case of two segments.

\begin{lemma}\label{lem:2segments} Let $A$ be an extremal set with cardinality $k$ composed by two segments. Then $A$ is isomorphic to one of the following sets:
\begin{enumerate}
	\item[(i)] $[0,k+b-1]\setminus [1,b]$, with $\dim (A)=1$, $\vol (A)=k+b$  and $|2A|=2k-1+b$ for some $1\le b\le k-3$, or
	\item[(ii)] $([0,k_1-1]\times \{0\})\cup ([0,k_2-1]\times \{1\})$ with $k_1+k_2=k$, $k_1,k_2\ge 1$, $\dim (A)=2$, $\vol (A)=k$  and  and $|2A|=3k-3$. 
\end{enumerate}
\end{lemma}

\begin{proof} If $\dim (A)=2$, then there must be no relation in the matrix $S_A$ in Corollary~\ref{cor:segdim} and we are led to Case (ii). 

Suppose that $\dim (A)=1$ and $A=P_1\cup P_2$, say  $P_1=[0,k_1-1]$ and $P_2=[k_1+\ell_1,k_1+\ell_1+k_2-1]$ for some $k_1, k_2\ge 1$ and $k=k_1+k_2$ and $\ell_1\ge 1$. Since $dim (A)=1$, we may also assume that $ (P_1+P_2)\cap 2P_2\neq \emptyset$, so that $2A$ consists of the interval $[0,2(k+\ell_1-1)]$ with a hole of some length $h\ge 0$. Since $A$ is extremal and has volume $vol(A)=k+\ell_1$ we have  $|2A|=2k-1+\ell_1$, so that $h=\ell_1$. Therefore
$$
\ell_1=\min(P_1+P_2)-\max (2P_1)-1=\ell_1+1-k_1,
$$
which implies $k_1=1$ and gives Case (i). 
\end{proof}

\begin{proof}[Proof of Theorem~\ref{thm:3seg}]
	Let $A$ consist of three segments $P_1$, $P_2$, and $P_3$, separated by intervals of holes $L_1$ and $L_2$. We consider three cases according to the dimension of $A$.
	
	{\bf Case~1.}  $\dim (A)=3$. By Corollary~\ref{cor:segdim} we must have ${\rm rank} (S_A) = 0$, that is all $P_i + P_j$ are disjoint for $1 \leq i,j \leq 3$, and we are led to Case~(iv) of the Theorem.
	
	{\bf Case~2.}  $\dim (A)=2$. It follows from Corollary~\ref{cor:segdim} that the matrix $S_A$  has ${\rm rank} (S_A) = 1$.  Up to isomorphisms we have two possibilities for the only independent relation in $S_A$.
	
	{\bf Case 2.1:} $(P_1+P_3)\cap 2P_2\neq \emptyset$. In this case we may assume that
	$$
	P_1=\{(0,0),\dots,(0,k_1-1)\},\; P_2=\{(1,0),\dots,(1,k_2-1)\}\; \text {and}\;  P_3=\{(2,\ell),\dots,(2,\ell+k_3-1)\}
	$$
	for some $\ell \in \N_0$. We know that $|2A| = 4k-6 - |(P_1+P_3) \cap 2P_2|$ so that $b = k-3- |(P_1+P_3) \cap 2P_2|$. We also have $\vol(A) = k + \max(\lfloor (k_1+\ell+k_3)/2 \rfloor -k_2,0)$. Since $\dim(A) = 2$ we must have $|(P_1+P_3) \cap 2P_2| > 0$ and therefore $\ell \leq 2k_2 - 2$. Now if $0 \leq \ell \leq 2k_2 - k_1 - k_3$ then $b = k-3 - (k_1+k_3-1) = k_2-2$ and $\vol(A) = k$. If $A$ is extremal, we therefore have $k_2 = 2$ so that $k_1+k_3 \leq 4$ and hence $k \leq 6$. If $\max(2k_2 - k_1 - k_3,0) < \ell \leq 2k_2 - 2$ then $b = k-3 - (2k_2 - 2 - \ell + 1) = k-2 + 2k_2 - \ell > \max(\lfloor (k_1+\ell+k_3)/2 \rfloor -k_2,0)$ so the set cannot be extremal. 
	
{\bf Case 2.2:}	$2P_1\cap (P_1+P_2) \neq \emptyset$. The case $(P_1+P_2) \cap 2P_2 \neq \emptyset$ works likewise. We may assume that
$$
P_1=  \{(0,0), (0,1),\cdots , (0,k_1-1)\},\; P_2=	\{(0,k_1+\ell_1), (0,k_1+\ell_1+1),\ldots ,(0,k_1+\ell_1+k_2-1)\},
$$
with $k_1\ge \ell_1+2$, and
$$
P_3=\{(1,0),\ldots ,(1,k_3-1)\}.
$$
	Let $A_0=P_0 \cup P_1$ and $A_1=P_2$, so that
	\begin{equation*}
		2A = 2A_0 \cup (A_0+A_1) \cup 2A_1,
	\end{equation*}
	the union being disjoint. We have $|2A_1| = 2k_3-1$ and, by Theorem~\ref{thm:3k-4}, we also have $|2A_0| \geq 2(k_1+k_2)-1+\ell_1$. Moreover, it can be readily checked that 
	$$
	|A_0+A_1|=\begin{cases} k_1+\ell_1 + k_2 + (k_3-1) = k + \ell_1 - 1, & \text{if} \; k_3>\ell_1+1,\\ (k_1+k_3-1) + (k_2+k_3-1) = k + k_3 - 2, & \text{otherwise}.\end{cases}
	$$
	It follows that 
	\begin{equation}\label{eq:case22}
	|2A| \geq 3k-3 + \ell_1 + \min\{k_3-1,\ell_1\}.
	\end{equation}
	As $\vol(A) = k + \ell_1$, the set can only be extremal if  $\min(k_3-1,\ell_1) = 0$, which implies $k_3=1$ (as $\ell_1\ge 1$) and there is equality in \eqref{eq:case22}, namely, if  $|2A_0| = 2(k_1+k_2) - 1 + \ell_1$. Applying Lemma~\ref{lem:2segments} to $A_0$ leads to Case (iii) of the Theorem. 

	{\bf Case~3.} $\dim (A)=1$. We recall the notation
	\begin{equation}
		|P_1|=k_1,\quad |P_2|=k_2,\quad |P_3|=k_3,\quad |L_1|=\ell_1,\quad |L_2|=\ell _2,
	\end{equation}
	and
 	\begin{equation}
 		a = \max (A)=k+\ell-1.
 	\end{equation}
	where $\ell = \ell_1+\ell_2$. The six segments in $2A$ are detailed below for further reference:
	\begin{align*}
		2P_1 &=  [0,2k_1-2],\\
		P_1+P_2 &= (k_1+\ell_1)+[0,k_1+k_2-2],\\
		2P_2 &= 2(k_1+\ell_1)+[0,2k_2-2],\\
		P_1+P_3 &= k_1+\ell_1+k_2+\ell_2+[0,k_1+k_3-2],\\
		P_2+P_3 & = 2(k_1+\ell_1)+(k_2+\ell_2)+[0,k_2+k_3-2],\\
		2P_3 &= 2(k_1+\ell_1+k_2+\ell_2)+[0,2k_3-2].
	\end{align*}
	Since $A$ is extremal, we have 
	\begin{equation}\label{eq:lba}
		a \geq 2(k+b-2),
	\end{equation}
	so that 
	\begin{equation}\label{eq:l1l2b}
		\ell \geq k+2b-3.
	\end{equation}
	We will use the following facts.

	\begin{claim}\label{claim:segment12}
		If $\max(k_1,k_2) < \ell_1 + 2$ then $2P_1$, $P_1+P_2$ and $2P_2$ are pairwise disjoint. If $\max (k_1,k_2) \ge \ell_1+2$ then $P_1 \cup P_2$ is $1$--dimensional and $2(P_1 \cup P_2)$ is a segment with a hole of length 
		\begin{equation*}
			h = \max\big\{ \ell_1  - \min (k_1,k_2)+1,0 \big\}.
		\end{equation*}
	\end{claim}
	
	\begin{proof}
		We note that $2P_1$ does not intersect $P_1+P_2$ if and only if  $\max (2P_1) < \min (P_1+P_2)$, which is equivalent to $k_1 < \ell_1+2$. Likewise, $P_1+P_2$ does not intersect $2P_2$ if and only if  $k_2 < \ell_1+2$, establishing the first part of the claim. 
		
		Assume without loss of generality that $k_1 \le k_2$ and  $k_2 \ge \ell_1+2$. Then $(P_1+P_2) \cup 2P_2$ is a segment since the two parts intersect. In particular,  $P_1 \cup P_2$ is $1$--dimensional. Moreover,  either $2P_1\cup (P_1+P_2)\cup 2P_2$ is a segment or a segment with a hole  of length $h=\min(P_1+P_2)-\max (2P_1)-1= \ell_1  - k_1+1$, establishing the second part of the claim. 
	\end{proof}
	
	By the above Claim,  if both $\max(k_1,k_2) < \ell_1 + 2$ and $\max(k_2,k_3) < \ell_2 + 2$, then the five segments in 
	\begin{equation*}
		2P_1\cup (P_1+P_2)\cup 2P_2 \cup (P_2+P_3) \cup 2P_3
	\end{equation*}
	are pairwise disjoint. Using Corollary~\ref{cor:segdim} it follows that ${\rm rank} (S_A) \leq 1$ and hence $\dim (A) \geq 2$, contradicting the assumption of this case.
	
	We will therefore without loss of generality assume that $\max (k_1,k_2) \geq \ell_1+2$. In this case we have
	
	\begin{claim}\label{claim:segment23}
	$\max \{k_2,k_3\}<\ell_2+2$.
	\end{claim}
	
	\begin{proof} Suppose on the contrary that $\max \{k_2,k_3\}\ge \ell_2+2$. Then, using \eqref{eq:l1l2b},
	$$
	k+2b-3\le \ell \le \max\{k_1,k_2\}+\max\{k_2,k_3\}-4
	$$
	which implies $\max\{k_1,k_2\}=\max\{k_2,k_3\}=k_2$. It follows that $\max (2P_2)=2(k_1+l_1+k_2-1)\ge 2(k_1+\ell_1)+k_2+\ell_2=\min (P_2+P_3)$. Hence, the sets $2(P_1\cup P_2)$ and $2(P_2\cup P_3)$ overlap and, by Claim~\ref{claim:segment12}, $2A$ consists of the interval $[0,2a]$ with two holes of total length at most 
	$$
	 \max \{ \ell_1-k_1+1,0\}+\max \{\ell_2-k_3+1,0\}\le \ell.
	$$
	Therefore, by using $a=k+\ell-1$ and \eqref{eq:l1l2b}, we obtain $|2A|\ge 2a-\ell+1\ge 3k+2b-4$ and therefore $A$ is not extremal.
	\end{proof}
	
	It follows from Claim~\ref{claim:segment12} and Claim~\ref{claim:segment23} that the three segments $2P_2, P_2+P_3, 2P_3$ are pairwise disjoint.
	Since $A$ is one--dimensional, $2(P_1\cup P_2)$ must intersect $P_1+P_3$. In particular, $\max (2P_2)\ge \min (P_1+P_3)$ which yields
	\begin{equation}\label{eq:ubl2}
	 k_1+\ell_1+k_2\ge \ell_2+2.
	\end{equation}

	\begin{claim}\label{claim:k3=1}
	$k_3=1$.
	\end{claim}
	
	\begin{proof} Suppose on the contrary that $k_3>1$. We then have $\ell_1>1$, since otherwise \eqref{eq:ubl2} and \eqref{eq:l1l2b} give $k_1+k_2\ge \ell_2+1\ge  k+2b-3$ and we get $k_3\le 1$.  
	
	Let $B=2(P_1\cup P_2) \cup (P_1+P_3)\cup (P_2+P_3)$. We can write $2A$ as the disjoint union
	$$
	2A=B \cup 2P_3.
	$$
	Consider now the set $A'$ obtained from $A$ by replacing $\min (P_3)$ with $\max (P_1)+1$ if $k_1\ge k_2$ and with $\min (P_2)-1$ otherwise. The resulting set is still composed of three disjoint segments, $A'=P'_1\cup P'_2\cup P'_3$ with $\ell'_1=\ell_1-1$, $\ell_2' = \ell_2 + 1$ and $\min \{k_1',k'_2\}=\min \{k_1,k_2\}$. We can   write $2A'$ as the disjoint union
	$$
	2A'=B' \cup 2P_3',
	$$
	where $B'=2(P_1'\cup P_2') \cup (P_1'+P_3')\cup (P_2'+P_3')$. We have $|2P'_3|=|2P_3|-2$. Let us show that $|B'|\le |B|+1$. 
	
	By Claim~\ref{claim:segment12}, $|2(P'_1\cup P'_2)|\le |2(P_1\cup P_2)|+1$. If $k_1\ge k_2$ then $P_2'=P_2$ and  $|P'_2+P'_3|=|P_2+P_3|-1$, while $P'_1+P'_3=(P_1+P_3)+1$. If $P_1+P_3$ and $P_2+P_3$ are disjoint then the two last modifications compensate each other, while if they intersect then there is no change in the cardinality of their union. Similarly, if $k_1<k_2$ then $P_1'=P_1$ and we loose one unit in $P'_1+P'_3$ while $P'_2+P'_3$ is translated one unit to the right from $P_2+P_3$ and again there is no change in the cardinality of the union of these two segments. 
			
In either case, we get $|2A'|<|2A|$ so that, if $A'$ is one--dimensional it would have the same volume as $A$ contradicting that  $A$ is extremal. It follows that $A'$ must be $2$--dimensional. This  implies $\max (2P'_2)<\min (P'_1+P_3')$. Since $\max (2P_2)\ge \min (P_1+P_3)$, we have equality in the last inequality.  Therefore,
	\begin{equation}
		|2A| = |2(P_1\cup P_2)|+|P_3+A|-1.
	\end{equation}
	By Theorem~\ref{thm:ls} we have 
	\begin{equation}
		|P_3+A| \ge |A|+2|P_3|-2=k+2k_3-2.
	\end{equation}
	Therefore, 
	\begin{align*}
		|2A| &= |2(P_1\cup P_2)|+|P_3+A|\\
			&\ge (\max(2P_2)+1-\ell_1)+(k+2k_3-2)\\
			&= 2(k-k_3+\ell_1-1)-\ell_1+k+2k_3-2\\
			&=3k+\ell_1-4,
	\end{align*}
	so that $\ell_1\le b$. But then, by \eqref{eq:ubl2}, we have 
	\begin{equation}
		\ell=\ell_1+\ell_2\le b+(k-k_3+b-2)=k-k_3+2b-2,
	\end{equation}
	contradicting \eqref{eq:l1l2b}. Therefore $A$ could not have been extremal.
	\end{proof}
	
	We can therefore assume $P_3=\{a\}$. It follows that
	$$
	2A=2(P_1\cup P_2)\cup (a+A).
	$$
	Moreover, 
	$$
	\min(P_2+P_3)-\max (P_1+P_3)=\ell_1-k_3+2=\ell_1+1> 1.
	$$

	We next consider two cases.
	
	{\bf Case 3.1:} $k_1\le k_2$. 
	
	The sumset  $2A$ can be written as the disjoint union 
	$$
	2A=B\cup (P_2+P_3)\cup 2P_3,
	$$
	where $B=2(P_1\cup P_2)\cup (P_1+P_3)$ is an interval with a hole of length $h=\max \{\ell_1-k_1+1,0\}$. Such a one--dimensional set with $k_1>1$ cannot be extremal since, by exchanging $\max (P_1)$ by $\min(P_2)-1$ we get a one--dimensional set with the same volume and smaller doubling. It follows that $k_1=1$. By using \eqref{eq:ubl2}, we get $\max (2P_2)-\max(P_1+P_3)=\ell_1+k_2-\ell_2-1\ge 0$. In this case $2(k+\ell_1-2)=\max (2P_2)\ge a = \max(P_1+P_3)$ and, again by extremality, equality holds. We thus have $|2A|=(a-\ell_1+1)+(k-2)+1=3k+\ell_1-4$, leading to Case (i) of the Theorem.

	{\bf Case 3.2:} $k_1>k_2$. 	
	
	From
	\begin{align*}
		3k-4+b=|2A|&=|2(P_1\cup P_2)|+|a+A|-|2(P_1\cup P_2)\cap (a+A)|\\
		&= 2(k-1+\ell_1)-1- \max\{ \ell_1 -k_2+1,0 \}+k \\
		& \quad -|2(P_1\cup P_2)\cap (a+A)|,
	\end{align*}
	we obtain
	\begin{equation}\label{eq:inta+A}
		|2(P_1\cup P_2)\cap (a+A)| = 2\ell_1+1 - \max\{ \ell_1 -k_2+1,0 \} - b.
	\end{equation}
	For this equality to hold, a necessary condition is 
	\begin{equation}\label{eq:condition}
	\max (2P_2)-a+1\ge 2\ell_1+1 - \max\{ \ell_1 -k_2+1,0 \} - b.
	\end{equation} 

	 By using $\max(2P_2)=2(k-1)+2\ell_1-2$ and $a=k+\ell-1$ in \eqref{eq:condition}, we obtain
	\begin{equation}\label{eq:l2}
		\ell \le  k+b+\max\{ \ell_1-k_2+1,0\}-3.
	\end{equation}
	By \eqref{eq:l1l2b}  we  have $\ell_1\ge k_2+b-1$. On the other hand, since $|2(P_1\cup P_2)\cap (a+A)| \leq |2P_2| = 2k_2-1$, it follows from \eqref{eq:inta+A} that $\ell_1\le b+k_2-1$. Hence $\ell_1=b+k_2-1$, there is equality in \eqref{eq:l2} and $2P_2$ must  be included in $P_1+P_3$, so that $2k_2-1 = |2P_2| \leq |P_1+P_3| =  k_1$. This gives  Case~(ii) of the Theorem.  \end{proof}

\end{document}